\documentclass{article}
\usepackage[utf8]{inputenc}
\usepackage{amssymb}
\usepackage{amsmath}
\usepackage{hyperref}
\usepackage{amsthm}
\usepackage{mathrsfs}
\usepackage{mathtools}
\usepackage{tikz}
\usepackage{bbm}
\usepackage{hyperref}
\usepackage{tikz-cd}
\usepackage{comment}
\usepackage{hieroglf}
\usepackage{fullpage}

\DeclareFontFamily{U}{rcjhbltx}{}
\DeclareFontShape{U}{rcjhbltx}{m}{n}{<->rcjhbltx}{}
\DeclareSymbolFont{hebrewletters}{U}{rcjhbltx}{m}{n}

\DeclareMathSymbol{\kaf}{\mathord}{hebrewletters}{107}

\newtheorem{theorem}{Theorem}[section]
\newtheorem{lemma}[theorem]{Lemma}

\newtheorem{proposition}[theorem]{Proposition}
\newtheorem{corollary}[theorem]{Corollary}
\theoremstyle{definition}
\newtheorem{example}[theorem]{Example}
\newtheorem{definition}[theorem]{Definition}
\newtheorem{remark}[theorem]{Remark}

\title{Derived control theorems for reductive groups}
\author{by Rob Rockwood}
\date{}

\begin{document}

\maketitle

\begin{abstract}
    We prove Hida-style control theorems in the derived setting for a large class of reductive groups tailored for applications to Euler systems. 
\end{abstract}
\section{Introduction}
Fix a prime $p$. Let $\mathcal{G}$ be a connected reductive algebraic group over $\mathbb{Q}$ unramified over $\mathbb{Q}_{p}$ with Borel subgroup $B_{G}$, splitting field $K/\mathbb{Q}_{p}$ and reductive model $G$ over $\mathbb{Z}_{p}$. Let $Q_{G}$ be a parabolic subgroup of $G$ with Levi decomposition $Q_{G} = L_{G} \times N_{G}$, where $N_{G}$ is the unipotent radical of $Q_{G}$ and $L_{G}$ is the Levi subgroup. Let $T_{G}$ be a maximal torus contained in $Q_{G}$ and let $S_{G} = L_{G}^{\mathrm{der}}\backslash L_{G}$. Write $S_{n}(\mathbb{Z}_{p}) \subset S_{G}(\mathbb{Z}_{p})$ be the subgroup of points which reduce to the identity mod $p^{n}$.  
 Let $\chi \in X^{\bullet}(L_{G})$, $\lambda \in X^{\bullet}(T_{G})$ be characters such that $\lambda$ is dominant for $B_{L_{G}}$ and $\lambda + \chi$ is dominant for $B_{G}$ and write $V_{\lambda + \chi}$ for the $K$-linear irreducible representation of $G$ of highest weight $\lambda$ and $W_{\lambda}$ for the $K$-linear irreducible representation of $L_{G}$ of highest weight $\lambda$. Write $V_{\lambda, \mathcal{O}_{K}}$ for the minimal admissible $\mathcal{O}_{K}$-lattice in $V_{\lambda + \chi}$ and $W_{\lambda,\mathcal{O}_{K}}$ for the minimal admissible lattice in $W_{\lambda}$. Let $\Gamma \subset G(\mathbb{Q}) \cap G(\hat{\mathbb{Z}})$ be a congruence subgroup of level prime to $p$, let $\Gamma_{0}(p^{n})$ be the subgroup of points which reduce to $Q_{G}(\mathbb{Z}/p^{n}\mathbb{Z}) \mod p^{n}$  and let $\Gamma_{1}(p^{n})$ be the subgroup of points which reduce to $N_{G}(\mathbb{Z}/p^{n}\mathbb{Z}) \mod p^{n}$.  
 
 We prove the following theorem:
\begin{theorem} 
For all $\lambda$ as above there is a perfect complex $M_{\lambda}^{\bullet} \in \mathscr{D}(\mathcal{O}_{K}[[S_{G}(\mathbb{Z}_{p})]])$ concentrated in degrees $[0,\nu]$ satisfying 
$$
H^{i}(M_{\lambda}^{\bullet}) = \varprojlim_{n} H^{i}(\Gamma_{1}(p^{n}), W_{\lambda,\mathcal{O}_{K}}/p^{n})^{\mathrm{ord}}
$$ 
and for all $\chi$ as above there is a quasi-isomorphism
$$
    M_{\lambda}^{\bullet} \otimes^{L}_{\mathcal{O}_{K}[[S_{n}(\mathbb{Z}_{p})]]} \mathcal{O}_{K}^{(\chi)} \sim R\Gamma(\Gamma_{1}(p^{n}), V_{\lambda + \chi, \mathcal{O}_{K}})^{\mathrm{ord}},
$$
for $n \geq 1$ and a quasi-isomorphism
$$
     M_{\lambda}^{\bullet} \otimes^{L}_{\mathcal{O}_{K}[[S_{G}(\mathbb{Z}_{p})]]} \mathcal{O}_{K}^{(\chi)} \sim R\Gamma(\Gamma_{0}(p^{n}), V_{\lambda + \chi, \mathcal{O}_{K}})^{\mathrm{ord}}
$$
for $n = 0$. 
\end{theorem}

The aim of this work is to provide a toolbox for those working with Euler systems varying in Hida families, such as those constructed in \cite{kings2015rankin} and \cite{LZGsp}, and to act as a companion piece to forthcoming work of Loeffler and Zerbes in which they construct such interpolating classes for the same broad class of reductive groups with which we work. 

We remark that there are many similar results in the literature, for example the work of Hida \cite{hidacontrol} for $\mathrm{SL}_{n}$ and Tilouine-Urban \cite{tilurb} for $\mathrm{GSp}_{4}$, albeit not in the derived setting. Indeed many of their proofs generalise readily to the general setting with only minor tweaks in order to work with complexes instead of cohomology groups and to account for changes in convention. The conventions in the aforementioned papers tend to differ greatly from those occurring in the literature on Euler systems and so we think it valuable, even in the existing cases, to have statements of these results with our conventions. 

The layout of the paper is as follows:
\begin{itemize}
    \item In Section \ref{sec:2} we fix the notations and conventions we will use for reductive groups, highest weight representations and interpolating modules. 
    \item In Section \ref{sec:3} we prove the derived control theorem. \item In Section \ref{sec:4} we prove $p$-stabilisation and duality results.
    \item In Section \ref{sec:5} we deduce control results for `adèlic cohomology'. We prove compatibility with the Hecke algebra $\mathbb{T}_{S,p}$ generated by the anemic Hecke algebra $\mathbb{T}_{S}$ and the $U_{p}$-operator and use this to prove a vanishing result for Iwasawa cohomology under the assumption that the Iwahori-level cohomology vanishes outside of the middle degree when localised at some maximal ideal of $\mathbb{T}_{S,p}$. 
\end{itemize}
\section{Notation} \label{sec:2}
\subsection{Algebraic groups and Iwasawa algebras}
The setting:
\begin{itemize}
\item $\mathcal{G}$ is a connected reductive algebraic $\mathbb{Q}$-group, unramified over $\mathbb{Q}_{p}$. The group-scheme $\mathcal{G}/\mathbb{Q}_{p}$ thus splits over a finite unramified extension $K$ of $\mathbb{Q}_{p}$ with ring of integers $\mathcal{O}$ and admits a reductive group-scheme model $G$ over $\mathbb{Z}_{p}$.
\item Fix a choice of Borel subgroup and maximal torus $B_{G} \supset T_{G}$ defined over $\mathbb{Z}_{p}$.
\item Fix a choice $Q_{G}$ of standard parabolic subgroup of $G$ with Levi factor $L_{G}$ and unipotent radical $N_{G}$.
Let $L_{G}^{\mathrm{der}}$ denote the derived subgroup of $L_{G}$. We write $\bar{Q}_{G}$ for the image of $Q_{G}$ under the longest Weyl element.
\item Let $S_{G} = L_{G}^{\mathrm{der}} \backslash L_{G}$.
\end{itemize}
Let $\eta: \mathbbm{G}_{m/\mathbb{Z}_{p}} \to Z(L_{G})$ be a cocharacter which is strictly dominant with respect to $Q_{G}$ in the sense that $\langle\eta, \Phi\rangle > 0$ for all relative roots $\Phi$. Set $\tau = \eta(p)$. Equivalently
$$
    \bigcap_{i}\tau^{-i}\bar{N}_{G}(\mathbb{Z}_{p})\tau^{i} = \{1\}.
$$
Define 
\begin{align*}
    N_{r} &= \tau^{r}N_{G}(\mathbb{Z}_{p})\tau^{-r},  \\
    \bar{N}_{r} &= \tau^{-r}N_{G}(\mathbb{Z}_{p})\tau^{r},  \\
    L_{r}^{\mathrm{der}} &=  \{ \ell \in L_{G}(\mathbb{Z}_{p}): \ell \ \mathrm{mod} \ p^{r} \in L_{G}^{\mathrm{der}}(\mathbb{Z}/p^{r}\mathbb{Z})\}
\end{align*}
for $r \geq 1$ and set $L_{0}^{\mathrm{der}} = L_{G}$ for future notational convenience.
Define level groups
\begin{align*}
    V_{0, r} &= \bar{N}_{r}L_{G}N_{0} \\
    V_{1,r} &= \bar{N}_{r}L^{\mathrm{der}}_{r}N_{0}.
\end{align*}
Fix a prime-to-$p$ congruence subgroup $\Gamma \subset G(\hat{\mathbb{Z}})$ and let 
$$
\Gamma_{?,r} = V_{?,r} \cap \Gamma,
$$
for $? \in \{0,1\}$.

Define
$$
    \Lambda_{0} := \mathcal{O}[[S_{G}(\mathbb{Z}_{p})]] 
$$
Let $S_{r}(\mathbb{Z}_{p}) = \{s \in S_{G}(\mathbb{Z}_{p}): s \equiv 1 \ \mathrm{mod} \ p^{r}\} = L_{G}^{\mathrm{der}}\backslash L_{r}^{\mathrm{der}}$. This is a free $\mathbb{Z}_{p}$-module. Set 
$$
    \Lambda_{r} := \mathcal{O}[[S_{r}(\mathbb{Z}_{p})]].
$$
The ring $\Lambda_{0}$ decomposes into a direct sum
$$
    \Lambda_{0} = \bigoplus_{\psi}\Lambda_{1}^{(\psi)},
$$
where the sum runs over characters $\psi: S_{1}(\mathbb{Z}_{p})\backslash S_{G}(\mathbb{Z}_{p}) \to \mathcal{O}^{\times}$ and $\Lambda_{1}^{(\psi)} = \Lambda_{1}$ with the action of $S_{1}(\mathbb{Z}_{p})\backslash S_{G}(\mathbb{Z}_{p}) $ given by $\psi$.

\begin{lemma} \label{lem:11}
Let $M$ be a  $\Lambda_{0}$-module. Suppose $M$ is free as a $\Lambda_{1}$-module under the inclusion $\Lambda_{1} \hookrightarrow \Lambda_{0}$. Then $M$ is projective as a $\Lambda_{0}$-module.
\end{lemma}

\begin{proof}
It suffices to prove that $\Lambda_{1}$ is a projective $\Lambda_{0}$-module, which is clear from the above decomposition.
\end{proof}

Given a character $\chi: S_{r}(\mathbb{Z}_{p}) \to \mathcal{O}^{\times}$ we write 
$$
\chi^{\dagger}: \Lambda_{0} \to \mathcal{O}
$$ 
for the induced homomorphism.

\subsection{Chain complexes and Hecke algebras} \label{sec:1b}
 Let $R$ be a ring. For an arithmetic subgroup $\Gamma \subset G(\mathbb{Z})$ we can find a resolution $\mathscr{C}_{\bullet}(\Gamma)$ of $R$ by finite free $R[\Gamma]$-modules \cite[Lemma 4.2.2]{urbaneigen}. Given a left $R[\Gamma]$-module $M$, we define a complex
$$
    \mathscr{C}^{\bullet}(\Gamma, M) := \mathrm{Hom}_{R[\Gamma]}(\mathscr{C}_{\bullet}(\Gamma), M),
$$
satisfying $H^{i}( \mathscr{C}^{\bullet}(\Gamma, M)) = H^{i}(\Gamma, M)$. This of course depends on the choice of $\mathscr{C}_{\bullet}(\Gamma)$ but its image $R\Gamma(\Gamma, M)$ in the derived category does not. 

Given groups $\Gamma, \Delta$, a $\Gamma$-module $M$ and $\Delta$-module $N$, it is a standard fact from group cohomology (see e.g. \cite[4.2.5]{urbaneigen}) that a pair $(\phi, f)$ consisting of a group homomorphism $\phi: \Gamma \to \Delta$ and a map of abelian groups $f: N \to M$ satisfying
$$
f(\phi(\gamma)m) = \gamma f(m) 
$$
for all $n \in N$ and $\gamma \in \Gamma$ induce a natural map 
$$
    \mathscr{C}^{\bullet}(\Delta, M) \to \mathscr{C}^{\bullet}(\Gamma, N) 
$$   
\begin{example}
\begin{itemize}
    \item If $\iota: \Gamma \hookrightarrow \Delta$ and $M = N$ then we have the restriction map 
    $$
        \mathrm{res}^{\Delta}_{\Gamma} = (\iota, \mathrm{id})
    $$
    \item If $\alpha \in G(\mathbb{Q})$ acts on $M$, then define 
    $$
        [\alpha] = (\alpha(\cdot)\alpha^{-1}, \alpha(\cdot))
    $$
\end{itemize}
\end{example}

If $\Gamma' \subset \Gamma$ is a finite index subgroup, then $\mathscr{C}_{\bullet}(\Gamma)$ is also a resolution of $R$ by finite free $\Gamma'$-modules so there is a unique homotopy equivalence $\delta: \mathscr{C}_{\bullet}(\Gamma) \to \mathscr{C}_{\bullet}(\Gamma')$ extending the identity. We define the corestriction map 
$$
    \mathrm{cores}^{\Gamma}_{\Gamma'}: \mathrm{Hom}_{\Gamma'}(\mathscr{C}_{\bullet}(\Gamma'), M) \xrightarrow{\circ\delta} \mathrm{Hom}_{\Gamma'}(\mathscr{C}_{\bullet}(\Gamma), M) \xrightarrow{\sum \gamma_{i}} \mathrm{Hom}_{\Gamma}(\mathscr{C}_{\bullet}(\Gamma), M),
$$
where $\gamma_{i}$ is a full set of coset representatives for $\Gamma'\backslash\Gamma$.

Now suppose we have arithmetic groups $\Gamma, \Gamma' \subset \Gamma''$ and that $M$ is an $R[\Gamma'']$-module with a compatible action of $\alpha \in G(\mathbb{Q})$. The double coset $\Gamma\alpha\Gamma'$ defines a map 
$$
\mathscr{C}^{\bullet}(\Gamma, M) \to \mathscr{C}^{\bullet}(\Gamma', M) 
$$
via
\begin{equation} \label{eq:13}
    [\Gamma\alpha \Gamma'] = \mathrm{cores}_{\Gamma' \cap \alpha^{-1}\Gamma\alpha}^{\Gamma'} \circ [\alpha] \circ \mathrm{res}^{\Gamma}_{\alpha\Gamma'\alpha^{-1} \cap \Gamma},
\end{equation}
where we suppress the dependence on the homotopy $\delta$.
\begin{definition}
Define 
$$
    \mathcal{T} = [\Gamma \tau^{-1} \Gamma].
$$
\end{definition}

\subsubsection{Ordinary subspaces}
We perform a somewhat ad-hoc construction of the ordinary subspace. Let $(R, \mathfrak{m})$ be a local ring complete with respect to the $\mathfrak{m}$-adic topology. Suppose $M$ is a topological $R[\Gamma]$-module with a compatible action of $\tau^{-1}$ and such that the action of $\mathcal{T}$ on $\mathscr{C}^{\bullet}(\Gamma, M)$ is continuous for the induced product topology. Suppose further that $M$ is compact so that $\mathscr{C}^{\bullet}(\Gamma, M)$ is also compact. Then we can make sense of the ordinary part of $\mathscr{C}^{\bullet}(\Gamma, M)$:
$$
    \mathscr{C}^{i}(\Gamma, M)^{\mathrm{ord}} := \bigcap_{ n\geq 0} \mathcal{T}^{n}\mathscr{C}^{i}(\Gamma, M).
$$
All of the coefficient modules that we consider will satisfy these conditions. 

Suppose we know that the quotients $ \mathscr{C}^{i}(\Gamma, M)/\mathfrak{m}^{n}$ are finite $R/\mathfrak{m}^{n}$-modules. Then 
by results of Pilloni \cite{pilloni} there is an idempotent $e = \lim_{n}\mathcal{T}^{n!}$ such that
$$
    e\mathscr{C}^{i}(\Gamma, M) = \mathscr{C}^{i}(\Gamma, M)^{\mathrm{ord}}.
$$

\subsection{Algebraic representations}
 Let $\lambda \in X^{\bullet}(T_{G})$ be dominant with respect to the Borel $B_{L_{G}} := B_{G} \cap L_{G}$ of $L_{G}$. Define
$$
    \mathscr{C}_{\mathrm{alg}}^{L_{G}}(\lambda) := \{f \in \mathcal{O}[L_{G/\mathcal{O}}]: f(bx) = (-\lambda)(b)f(x) \ \forall \ b \in B_{L_{G}/\mathcal{O}}\},
$$ 
an admissible $\mathcal{O}$-lattice (in the sense of \cite[4.2]{LZGsp}) in the $K$-linear irreducible representation of $L_{G/K}$ of lowest weight $-\lambda$ with left $L_{G}$ action given by right translation. Suppose $\chi \in X^{\bullet}(L_{G})$ is such that $\lambda + \chi$ is dominant for $G$, and write 
\begin{align*}
\mathscr{C}_{\mathrm{alg}}^{G}(\lambda + \chi) &= \{f \in \mathcal{O}[G_{/\mathcal{O}}] \otimes  \mathscr{C}_{\mathrm{alg}}^{L_{G}}(\lambda): f(qg) = (-\chi(q))qf(g) \ \forall \ q \in Q_{G/\mathcal{O}}\} \\
&\cong  \{f \in \mathcal{O}[G_{/\mathcal{O}}]: f(bg) = (-\lambda - \chi)(b)f(g) \ \forall \ b \in B_{G/\mathcal{O}}\}
\end{align*}
an admissible $\mathcal{O}$-lattice in the irreducible representation of $G_{/K}$ of lowest weight $-(\lambda + \chi)$. The above isomorphism is given by mapping the $\mathscr{C}_{\mathrm{alg}}^{L_{G}}(\lambda)$ factor to $K$ under the `evaluation at 1' map.

We take $f^{\mathrm{lw}}_{\lambda}$ to be the unique lowest weight vector satisfying $f^{\mathrm{lw}}_{\lambda}(\bar{n}mn) = \lambda(m)$ for $\bar{n}mn \in \bar{N}_{G}L_{G}N_{G}$. 

\begin{definition}
Define 
\begin{align*}
    W_{\lambda, \mathcal{O}} &= \mathrm{Hom}_{\mathcal{O}}(\mathscr{C}_{\mathrm{alg}}^{L_{G}}(\lambda), \mathcal{O}) \\
    V_{\lambda + \chi, \mathcal{O}} &= \mathrm{Hom}_{\mathcal{O}}( \mathscr{C}_{\mathrm{alg}}^{G}(\lambda + \chi), \mathcal{O}).
\end{align*}
given the structure of $L_{G/\mathcal{O}}$ (resp. $G_{/\mathcal{O}}$) representations via the contragredient representation. We note that $W_{\lambda,\mathcal{O}}$ is an admissible lattice in the $K$-linear representation of $L_{G}$ of highest weight $\lambda$ and $V_{\lambda + \chi, \mathcal{O}}$ is an admissible lattice in the $K$-linear highest weight representation of highest weight $\lambda + \chi$. 
\end{definition}

We define an action of $\tau$ on $V_{\lambda, \mathcal{O}}$ as follows: $\tau$ gives a well-defined map $V_{\lambda, \mathcal{O}} \otimes K \to V_{\lambda, \mathcal{O}} \otimes K$ and if we set $h_{\lambda} = \langle \eta, \lambda \rangle$, then $p^{h_{\lambda}}\tau^{-1}$ preserves the lattice $V_{\lambda, \mathcal{O}}$, so we let 
$$
    \tau^{-1} * v = p^{h_{\lambda}}\tau^{-1}v.
$$
\begin{remark}
This action corresponds to the action of $\tau$ on $\mathscr{C}_{\mathrm{alg}}^{G}(\lambda)$
given by restricting to the big Bruhat cell $\bar{N}_{G}L_{G}N_{G}$ and setting $(\tau * f)(\bar{n}\ell n) = f(\tau^{-1}\bar{n}\tau \ell n)$.
\end{remark}
\subsection{Distribution modules}
We define the distribution modules which will serve as the coefficients for our interpolating complexes. 

Define spaces 
\begin{align*}
  Y_{r} &= L_{r}^{\mathrm{der}}N_{0}\backslash V_{0,1} \cong  L_{r}^{\mathrm{der}}\backslash L_{G}(\mathbb{Z}_{p}) \times \bar{N}_{1} , \\ 
    Y_{\mathrm{univ}} &=    L_{G}^{\mathrm{der}}N_{0}\backslash V_{0,1} \cong  (L_{G}^{\mathrm{der}}\backslash L_{G})(\mathbb{Z}_{p}) \times \bar{N}_{1}.
\end{align*}

We extend the natural right action of $V_{0,1}$ on $Y_{r}$ to an action of the monoid generated by $V_{0,1}$ and $\tau$ by letting 
$$
    (\ell,n) * \tau = (\ell, \tau^{-1} n\tau).
$$
Given a character $\lambda \in X^{\bullet}(T_{G})$ dominant with respect to $B_{L_{G}}$, let
\begin{align*}  
    \mathscr{C}_{r}(\lambda) &= \{\text{Continuous} \ f: N_{G}(\mathbb{Z}_{p}) \backslash V_{0,1} \to \mathscr{C}_{\mathrm{alg}}^{L_{G}}(\lambda): f(\ell x) = \ell f(x) \ \forall \ \ell \in L_{r}^{\mathrm{der}}\}, \\
    \mathscr{C}_{\mathrm{univ}}(\lambda) &= \{\text{Continuous} \ f: N_{G}(\mathbb{Z}_{p}) \backslash V_{0,1} \to \mathscr{C}_{\mathrm{alg}}^{L_{G}}(\lambda): f(\ell x) = \ell f(x) \ \forall \ \ell \in L_{G}^{\mathrm{der}}\}
\end{align*}
We endow these spaces with an action of $V_{0,1}$ by right translation:
$$
    (g\cdot f)(x) = f(xg).
$$
We define a twisted action of $L_{G}(\mathbb{Z}_{p})$ on $ \mathscr{C}_{\mathrm{univ}}(\lambda)$:
$$
    (\ell\cdot f)(x) = \ell f(\ell^{-1}x).
$$
This action is trivial on $ L_{G}^{\mathrm{der}}(\mathbb{Z}_{p})$ and extends to an action of $\Lambda_{0}$ (this is well-defined as $L_{G}$ normalises $N_{G}$).
Define
\begin{align*}
   \mathbb{D}_{r}(\lambda) = \mathrm{Hom}_{\mathcal{O}, \mathrm{cont}}( \mathscr{C}_{r}(\lambda), \mathcal{O}), \\
    \mathbb{D}_{\mathrm{univ}}(\lambda) = \mathrm{Hom}_{\mathcal{O}, \mathrm{cont}}( \mathscr{C}_{\mathrm{univ}}(\lambda), \mathcal{O}).
\end{align*}
There are isomorphisms of $\mathcal{O}$-modules:
\begin{equation} \label{eq:11}
      \mathbb{D}_{r}(\lambda)  \cong \mathcal{O}[[Y_{r}]] \otimes W_{\lambda, \mathcal{O}},
\end{equation}
\begin{equation} \label{eq:12}
      \mathbb{D}_{\mathrm{univ}}(\lambda)  \cong \mathcal{O}[[Y_{\mathrm{univ}}]] \otimes W_{\lambda, \mathcal{O}}.
\end{equation}
Let $\chi \in X^{\bullet}(S_{G})$ be a character such that $\lambda + \chi$ is dominant for $G$. There is a natural map
$$
    \mathbb{D}_{\mathrm{univ}}(\lambda) \to \mathbb{D}_{r}(\lambda + \chi)
$$
factoring through $\mathbb{D}_{\mathrm{univ}}(\lambda) \otimes \mathcal{O}^{(\chi)}$, 
given by dualising the inclusion 
$$
    \mathscr{C}_{r}(\lambda + \chi) \hookrightarrow \mathscr{C}_{\mathrm{univ}}(\lambda).
$$

In our proof of the control theorem we will need a few finite modules.
\begin{definition}
Set $\mathcal{O}_{s} := \mathcal{O}/p^{s}$ and $\mathscr{C}_{\mathrm{alg}}^{L_{G}}(\lambda; p^{s}) := \mathscr{C}_{\mathrm{alg}}^{L_{G}}(\lambda)\otimes\mathcal{O}_{s}$. For $s \geq r$ we define the following $\mathcal{O}_{s}$-modules:
\begin{align*}
    \mathscr{C}_{r}(\lambda ; p^{s}) &:= \{f: N_{G}(\mathbb{Z}/p^{s}\mathbb{Z}) \backslash V_{0,1}(p^{s}) \to W_{\lambda,s}: f(x\ell) = \ell f(x) \ \forall \ell \in L_{r}^{\mathrm{der}} \} \\
    \mathbb{D}_{r}(\lambda ; p^{s}) &:= \mathrm{Hom}_{\mathcal{O}_{s}}( \mathscr{C}_{r}(\lambda ; p^{s}), \mathcal{O}_{s}) \\
     \tilde{\mathscr{C}}_{r}(\lambda ; p^{s}) &:= \{f: L_{G}(\mathbb{Z}/p^{s}\mathbb{Z}) \to \mathscr{C}_{\mathrm{alg}}^{L_{G}}(\lambda; p^{s}): f(x\ell) = \ell f(x) \ \forall \ell \in L_{r}^{\mathrm{der}}\} \\
      \tilde{\mathbb{D}}_{r}(\lambda ; p^{s}) &:= \mathrm{Hom}_{\mathcal{O}_{s}}( \tilde{\mathscr{C}}_{r}(\lambda ; p^{s}), \mathcal{O}_{s}),
\end{align*}
where $V_{0,1}(p^{s}) \subset G(\mathbb{Z}/p^{s}\mathbb{Z})$ is the mod $p^{s}$ reduction of $V_{0,1}$. We endow $\mathbb{D}_{r}(\lambda;p^{s})$ with the action of $\Gamma_{0,1}$ corresponding to right translation of functions and give $\tilde{\mathbb{D}}_{r}(\lambda;p^{s})$ an analogous action of $\Gamma_{0,s}$.
\end{definition}
The utility of these modules is given by the fact that
$$
     \tilde{\mathbb{D}}_{r}(\lambda ; p^{s}) = \mathrm{Ind}_{\Gamma_{1,r} \cap \Gamma_{0,s}}^{\Gamma_{0,s}}W_{\lambda, s},
$$
and 
$$
    \varprojlim_{s} \mathbb{D}_{r}(\lambda; p^{s}) = \mathbb{D}_{r}(\lambda).
$$
\section{Derived control} \label{sec:3}
Let $\nu$ be the virtual cohomological dimension of $G$. For a commutative ring $R$ let $\mathscr{D}(R)$ denote the derived category of $R$-modules. 
\begin{definition}
A bounded complex of $R$-modules is called \textit{perfect} if it consists of finite projective $R$-modules. We call an object $M \in \mathscr{D}(R)$ perfect if it can be lifted to a perfect complex of $R$-modules.
\end{definition}

Write $\mathcal{O}^{(\chi)}$ for $\mathcal{O}$ with $\Lambda_{0}$-module structure given by $\chi^{\dagger}$.
We prove the following theorem. 
\begin{theorem} \label{thm:5}
For each $\lambda \in X^{\bullet}(T_{G})$ dominant for $B_{L_{G}}$ there is a perfect complex $M_{\lambda}^{\bullet} \in \mathscr{D}(\Lambda_{0})$ concentrated in degrees $[0,\nu]$ satisfying 
$$
H^{i}(M_{\lambda}^{\bullet}) = \varprojlim_{r} H^{i}(\Gamma_{1,r}, W_{\lambda,
\mathcal{O}}\otimes\mathcal{O}_{r})^{\mathrm{ord}}
$$ 
and for all $\chi \in X^{\bullet}(S_{G})$  such that $\lambda + \chi$ is dominant for $G$ there are a quasi-isomorphisms
$$
    M_{\lambda}^{\bullet} \otimes^{L}_{\Lambda_{r}} \mathcal{O}^{(\chi)} \sim R\Gamma(\Gamma_{1,r}, V_{\lambda + \chi, \mathcal{O}})^{\mathrm{ord}}
$$
for $r \geq 1$ and 
$$
     M_{\lambda}^{\bullet} \otimes^{L}_{\Lambda_{0}} \mathcal{O}^{(\chi)} \sim R\Gamma(\Gamma_{0,1}, V_{\lambda + \chi, \mathcal{O}})^{\mathrm{ord}}
$$
for $r = 0$. 
\end{theorem}
\begin{remark}
This result corrects an error in the main result of \cite{AshSteve} in which a small indexing mistake in the proof of Lemma 1.1 hides the contribution of some inscrutable Tor groups to the kernel of the specialisation map. Explicitly, it is stated in \textit{op.cit} that for $G = \mathrm{GL}_{n}$ there is an injective map
$$
    H^{i}(\Gamma_{0,1}, \mathbb{D}_{\mathrm{univ}})^{\mathrm{ord}}/I^{(\lambda)}_{0} \hookrightarrow H^{i}(\Gamma_{0,1}, V_{\lambda, \mathcal{O}})^{\mathrm{ord}}
$$
for all $i$. Our analysis shows that for this to happen it is necessary that the image of the Tor group $\mathrm{Tor}_{-2}^{\Lambda_{0}}( H^{i - 1}(\Gamma_{0,1}, \mathbb{D}_{\mathrm{univ}})^{\mathrm{ord}}, \mathcal{O}^{(\lambda)})$ in $H^{i}(\Gamma_{0,1}, \mathbb{D}_{\mathrm{univ}})^{\mathrm{ord}}/I^{(\lambda)}_{0}$ vanishes. It seems to us that there is no \textit{a priori} reason that this should be the case- it is not purely formal from the numerology. We describe in Section \ref{sec:5} some additional hypothesis that force the vanishing of the Tor groups which constitute the obstruction to the statement in \textit{op. cit}. 
\end{remark}
\subsection{Control} 
In this section we prove that there is a sequence of quasi-isomorphisms for $r \geq 0$:
$$
    \mathscr{C}^{\bullet}(\Gamma_{0,1}, \mathbb{D}_{\mathrm{univ}}(\lambda))^{\mathrm{ord}} \otimes_{\Lambda_{r}} \mathcal{O}^{(\chi)} \sim \mathscr{C}^{\bullet}(\Gamma_{0,1}, \mathbb{D}_{r}(\lambda + \chi))^{\mathrm{ord}} \sim \mathscr{C}^{\bullet}(\Gamma_{1,r}, V_{\lambda + \chi, \mathcal{O}})^{\mathrm{ord}}.
$$

\begin{definition}
Let $\{s^{(r)}_{i}\}_{i}$ be a $\mathbb{Z}_{p}$-basis for $T_{r}(\mathbb{Z}_{p})$. Given an algebraic character $\chi : S_{G}(\mathbb{Z}_{p}) \to \mathcal{O}^{\times}$ we write $I^{(\chi)}_{r}$ for the kernel of the induced homomorphism
$$
    \chi^{\dagger}: \Lambda_{r} \to \mathcal{O}.
$$ 
It is generated by the regular sequence $([s^{(r)}_{i}] - \chi(s^{(r)}_{i}))_{i}$.
\end{definition}
The first of the above sequence of quasi-isomorphisms is an immediate consequence of the following lemma.
\begin{lemma} \label{lem:4}
The following are equivalent:
\begin{itemize}
    \item $\mu \in I^{(\chi)}_{r}\mathbb{D}_{\mathrm{univ}}(\lambda)$
    \item The image of $\mu$ under the map
    $$
        \mathbb{D}_{\mathrm{univ}}(\lambda) \to \mathbb{D}_{r}(\lambda + \chi)
    $$
    is zero.
\end{itemize}
\end{lemma}
\begin{proof}
The space $\mathscr{C}_{r}(\lambda + \chi)$ can be realised as the subspace of $\mathscr{C}_{\mathrm{univ}}(\lambda)$ satisfying $(\ell \cdot f)(x) = \chi(\ell)f(x)$ for all $\ell \in  L_{r}^{\mathrm{der}}$. 
This inclusion dualises to the restriction map 
$$
    \mathbb{D}_{\mathrm{univ}}(\lambda) \to \mathbb{D}_{r}(\lambda + \chi).
$$
We can identify $ \mathscr{C}_{\mathrm{univ}}(\lambda)$ with the space of continuous $W_{\lambda, \mathcal{O}}$-valued functions on $Y_{\mathrm{univ}}$ by sending 
$$
    f \mapsto (\tilde{f}: x \mapsto x^{-1}f(x)).
$$
If for $\ell \in L_{G}(\mathbb{Z}_{p})$ we set $(\ell \cdot \tilde{f})(x) = \tilde{f}(\ell^{-1}x)$ then this isomorphism is equivariant for the action of $\Lambda_{0}$. Under this isomorphism $\mathscr{C}_{r}(\lambda + \chi)$ is the space of functions $g$ satisfying $g(\ell x) = \chi(\ell)g( x)$ for all $\ell \in L_{r}^{\mathrm{der}}$. This allows us to write $g \in \mathscr{C}_{r}(\lambda + \chi)$  as a product $\chi\tilde{g}$
where 
\begin{align*}
\tilde{g}: \bar{N}_{1} \times L_{r}^{\mathrm{der}} \backslash L_{G}(\mathbb{Z}_{p}) &\to \mathscr{C}_{\mathrm{alg}}^{L_{G}}(\lambda) \\
\chi: L_{G}^{\mathrm{der}}(\mathbb{Z}_{p}) \backslash L_{r}^{\mathrm{der}} &\to \mathcal{O}^{\times}.
\end{align*}

Note that $I_{r}^{(\chi)}$ is the kernel of the map 
$$
    \mu \mapsto \int_{L_{G}^{\mathrm{der}} \backslash L_{r}^{\mathrm{der}}}\chi(x)\mu(x).
$$ 

Suppose $\mu \in \mathbb{D}_{\mathrm{univ}}(\lambda)$ restricts to zero and suppose first that $\mu$ is a non-zero tensor $\mu = \mu_{1} \otimes \mu_{2}$. Then for $g \in  \mathscr{C}_{r}(\lambda + \chi)$
\begin{align*}
    \int_{Y_{\mathrm{univ}}}g(y)\mu(y) &= \int_{L_{G}^{\mathrm{der}} \backslash L_{r}^{\mathrm{der}}}\chi(\ell)\mu_{2}(\ell) \cdot \int_{\bar{N}_{1}\times L_{r}^{\mathrm{der}} \backslash L_{G}}\tilde{g}(\bar{n}\ell)\mu_{1}(\bar{n}\ell) \\
    &= 0
\end{align*}
for all $g$, which implies $\mu \in I_{r}^{(\chi)}\mathbb{D}_{\mathrm{univ}}(\lambda)$. In general we can see from \eqref{eq:11} that the space $\mathbb{D}_{\mathrm{univ}}(\lambda)$ has a Banach basis $\{\mu^{(i)}\}$ where each $\mu^{(i)}$ splits as a tensor product $\mu^{(i)} = \mu^{(i)}_{1} \otimes \mu^{(i)}_{2}$ so for general $\mu$ we can expand as a power series in $\mu^{(i)}$ and apply the above argument to each term to see that $\mu \in I_{r}^{(\chi)}\mathbb{D}_{\mathrm{univ}}(\lambda)$.

The converse is easy.
\end{proof}
 The second quasi-isomorphism follows from the following few lemmas. The next lemma is a variation of a lemma which appears frequently in papers on Hida theory, for example \cite[Proposition 4.1]{hidacontrol}, \cite[Lemma 3.1]{tilurb}, and shall be henceforth known as the `Hida lemma'.
\begin{lemma}
For $s \geq r$ let $M$ be a compact $\Gamma_{1,r}$-module with a compatible action of $\tau^{-1}$. The following diagram commutes on cohomology
$$
\begin{tikzcd}
\mathscr{C}^{\bullet}(\Gamma_{1,r} \cap \Gamma_{0,s}, M) \arrow[d, "\mathcal{T}^{s - r}"] \arrow[r, "\mathrm{cores}"] & \mathscr{C}^{\bullet}(\Gamma_{1,r}, M) \arrow[dl, "\tau^{-(s - r)}"] \arrow[d, "\mathcal{T}^{s - r}"] \\
\mathscr{C}^{\bullet}(\Gamma_{1,r} \cap \Gamma_{0,s}, M) \arrow[r, "\mathrm{cores}"'] & \mathscr{C}^{\bullet}(\Gamma_{1,r}, M)
\end{tikzcd}
$$
\end{lemma}
\begin{proof}
We do the proof for the top triangle, the bottom triangle being similar. We first note that 
$$
    (\Gamma_{1,r} \cap \Gamma_{0,s}) \cap \tau^{-(s - r)} (\Gamma_{1,r} \cap \Gamma_{0,s} )\tau^{s - r} = \bar{N}_{2s - r}L_{r}N_{0}
$$
so that when computing $\mathcal{T}^{s - r}$ our corestriction will sum over representatives for $\bar{N}_{s}/\bar{N}_{2s - r}$. Then
$$
    \Gamma_{1,r} / \Gamma_{1,r} \cap \Gamma_{0,s} = \bar{N}_{r}/\bar{N}_{s},
$$
and so if $\gamma_{i}'$ is s set of representatives for $  \Gamma_{1,r} / \Gamma_{1,r} \cap \Gamma_{0,s}$, then $\gamma_{i} := \tau^{-(s - r)}\gamma_{i}'\tau^{s - r}$ is a set of representatives for $\bar{N}_{s}/\bar{N}_{2s - r}$.

The Hecke operator $\mathcal{T}$ is given on complexes by the composition of pairs
\begin{align*}
    (\delta_{1}, \sum_{i} \gamma_{i}) \circ (\tau^{-1}(\cdot)\tau, \tau^{-1}) \circ (\iota, \mathrm{id}) &= (\tau^{-1}\delta_{1}(\cdot)\tau, \sum \gamma_{i}\tau^{-1})
\end{align*}
where $\delta_{1}$ is the canonical homotopy equivalence
$$
\mathscr{C}_{\bullet}( (\Gamma_{1,r} \cap \Gamma_{0,s}) \cap \tau (\Gamma_{1,r} \cap \Gamma_{0,s} )\tau^{-1}) \to \mathscr{C}_{\bullet}(\Gamma_{1,r} \cap \Gamma_{0,s})
$$
extending the identity and abuse notation by writing corestriction as a compatible pair,
and $\tau^{-1} \circ \mathrm{cores}$ is given by 
\begin{align*}
    (\tau^{-1}(\cdot)\tau, \tau^{-1}) \circ (\delta_{2}, \sum \gamma_{i}') &= (\delta_{2}(\tau^{-1}(\cdot)\tau), \sum \tau^{-1}\gamma_{i}') \\
    &= (\delta_{2}(\tau^{-1}(\cdot)\tau), \sum \gamma_{i}\tau^{-1}),
\end{align*}
where $\delta_{2}$ is the canonical homotopy equivalence of $\Gamma_{1,r} \cap \Gamma_{0,s}$-complexes
$$
   \mathscr{C}_{\bullet}(\Gamma_{1,r} \cap \Gamma_{0,s}) \to \mathscr{C}_{\bullet}(\Gamma_{1,r}).
$$
The maps $\delta_{i}$ induce the identity on cohomology so the result follows.
\end{proof}
The upshot of this lemma is that the restriction of the above corestriction map to the ordinary subspace is a quasi-isomorphism, so we have a quasi-isomorphism
$$
    \mathscr{C}^{\bullet}(\Gamma_{1,r} \cap \Gamma_{0,s}, M)^{\mathrm{ord}} \cong  \mathscr{C}^{\bullet}(\Gamma_{1,r}, M)^{\mathrm{ord}}.
$$

The following proposition will allow us to attack our problem using the Hida lemma in conjunction with Shapiro's lemma. Let $V_{\lambda + \chi, s} := V_{\lambda + \chi, \mathcal{O}} \otimes \mathcal{O}_{s}$, $W_{\lambda, s} := W_{\lambda, \mathcal{O}} \otimes \mathcal{O}_{s}$.

\begin{proposition} \label{prop:6}
Let $s \geq r$. There are isomorphisms 
\begin{align*}
    \mathscr{C}^{\bullet}(\Gamma_{1,r} \cap \Gamma_{0,s}, V_{\lambda + \chi, s})^{\mathrm{ord}} &\cong  \mathscr{C}^{\bullet}(\Gamma_{1,r} \cap \Gamma_{0,s}, W_{\lambda , s}(\chi))^{\mathrm{ord}} \\
      \mathscr{C}^{\bullet}(\Gamma_{0,s}, \mathbb{D}_{r}(\lambda; p^{s}))^{\mathrm{ord}} & \cong  \mathscr{C}^{\bullet}(\Gamma_{0,s}, \tilde{\mathbb{D}}_{r}(\lambda; p^{s}))^{\mathrm{ord}}
\end{align*}
\end{proposition}
\begin{proof}
We prove the first isomorphism, the second being not dissimilar. Consider the map 
$$
      \mathscr{C}^{\bullet}(\Gamma_{1,r} \cap \Gamma_{0,s}, V_{\lambda + \chi, s})^{\mathrm{ord}} \to \mathscr{C}^{\bullet}(\Gamma_{1,r} \cap \Gamma_{0,s}, W_{\lambda, s}(\chi))^{\mathrm{ord}} 
$$
induced by the dual of the inclusion of $L_{G}$-modules 
$$
   \mathscr{C}_{\mathrm{alg}}^{L_{G}}(\lambda; p^{s})(\chi^{-1}) \hookrightarrow \mathscr{C}_{\mathrm{alg}}^{G}(\lambda + \chi; p^{s}) .
$$
Regular functions on $G$ are uniquely defined by their restriction to the big Bruhat cell $\bar{N}_{G}L_{G}N_{G}$ and we can view the image of $   \mathscr{C}_{\mathrm{alg}}^{L_{G}}(\lambda)(\chi^{-1})$ in $\mathscr{C}_{\mathrm{alg}}^{G}(\lambda + \chi)$ under the above isomorphism as being the functions $f$ whose restriction to the big Bruhat cell satisfy $f(\bar{n}\ell n) = f(\bar{n}\ell)$. 
The kernel $\mathcal{K}_{s}$ of the dual map is given by functionals
$$
\mathscr{C}_{\mathrm{alg}}^{G}(\lambda + \chi; p^{s})  \to \mathcal{O}_{s}
$$
whose restriction to $ \mathscr{C}_{\mathrm{alg}}^{L_{G}}(\lambda; p^{s})(\chi^{-1})$ is zero. We want to show that 
$$
    \mathscr{C}^{i}(\Gamma_{1,r} \cap \Gamma_{0,s}, \mathcal{K}_{s})^{\mathrm{ord}} = 0.
$$

Let $z \in  \mathscr{C}^{i}(\Gamma_{1,r} \cap \Gamma_{0,s}, V_{\lambda + \chi, s})$, then for $c \in \mathscr{C}_{i}(\Gamma)$ there are functionals $\phi^{(c)}_{i} \in V_{\lambda + \chi, s}$ such that 
$$
    (\mathcal{T}^{k}z)(c) = \sum \gamma_{i}\tau^{-k}*\phi^{(c)}_{i},
$$
so via the contragredient representation it suffices to show that for $f \in \mathscr{C}_{\mathrm{alg}}^{G}(\lambda + \chi)$ and $k \gg 0$ we have $(\tau^{k} * f)(\bar{n}\ell n) = (\tau^{k} * f)(\bar{n}\ell)$, but $\tau_{-k}\bar{N}_{0}\tau^{k} = \bar{N}_{k}$, so for $k \geq s$
$$
    \tau^{-k}\bar{N}_{0}\tau^{k} \equiv 1 \ \mathrm{mod} \ p^{s},
$$
whence we are done.
\end{proof}
\begin{lemma} \label{lem:5}
For $r \geq 1$ there is a quasi-isomorphism
$$
    \mathscr{C}^{\bullet}(\Gamma_{0,1}, \mathbb{D}_{r}(\lambda + \chi))^{\mathrm{ord}} \cong \mathscr{C}^{\bullet}(\Gamma_{1,r} ,   V_{\lambda + \chi,\mathcal{O}})^{\mathrm{ord}}.
$$
\end{lemma}
\begin{proof}
We have the following chain of quasi-isomorphism
\begin{align*}
    \mathscr{C}^{\bullet}(\Gamma_{0,1}, \mathbb{D}_{r}(\lambda + \chi))^{\mathrm{ord}} &\cong  \varprojlim_{s}  \mathscr{C}^{\bullet}(\Gamma_{0,1}, \mathbb{D}_{r}(\lambda + \chi; p^{s}))^{\mathrm{ord}} \\
    &\cong \varprojlim_{s} \mathscr{C}^{\bullet}(\Gamma_{0,s}, \mathbb{D}_{r}(\lambda + \chi; p^{s}))^{\mathrm{ord}} \\
    &\cong \varprojlim_{s} \mathscr{C}^{\bullet}(\Gamma_{0,s}, \tilde{\mathbb{D}}_{r}(\lambda + \chi; p^{s}))^{\mathrm{ord}} \\
    &\cong \varprojlim_{s} \mathscr{C}^{\bullet}(\Gamma_{1,r} \cap \Gamma_{0,s}, W_{\lambda,s}(\chi))^{\mathrm{ord}} \\
    &\cong \varprojlim_{s} \mathscr{C}^{\bullet}(\Gamma_{1,r} \cap \Gamma_{0,s}, V_{\lambda + \chi,s})^{\mathrm{ord}} \\
    &\cong \varprojlim_{s} \mathscr{C}^{\bullet}(\Gamma_{1,r}, V_{\lambda + \chi,s})^{\mathrm{ord}} \\
    &\cong \mathscr{C}^{\bullet}(\Gamma_{1,r}, V_{\lambda + \chi, \mathcal{O}})^{\mathrm{ord}},
\end{align*}
where the second line is the Hida lemma, the fourth line is Shapiro's lemma and the sixth line is the Hida lemma again.
\end{proof}
The case $r = 0$ has the same proof with the Shapiro's lemma step being trivial.
\begin{corollary}
There is a quasi-isomorphism
$$
    \mathscr{C}^{\bullet}(\Gamma_{0,1}, \mathbb{D}_{\mathrm{univ}}(\lambda))^{\mathrm{ord}} \cong \varprojlim_{r} \mathscr{C}^{\bullet}(\Gamma_{1,r}, W_{\lambda, r})^{\mathrm{ord}}.
$$
\end{corollary}
\begin{proof}
Note that $\mathbb{D}_{\mathrm{univ}}(\lambda) = \varprojlim_{r} \mathbb{D}_{r}(\lambda) = \varprojlim_{r,s}\mathbb{D}_{r}(\lambda;p^{s})$. The previous lemma gives us that 
$$
    \mathscr{C}^{\bullet}(\Gamma_{0,1}, \mathbb{D}_{r}(\lambda;p^{s}))^{\mathrm{ord}} = \mathscr{C}^{\bullet}(\Gamma_{1,r}, W_{\lambda, s})^{\mathrm{ord}}
$$
and the result follows from taking the inverse limit over $r,s$ on both sides. 
\end{proof}

We write 
$$
\mathscr{C}^{i}(\Gamma_{1, \infty}, W_{\lambda, \mathcal{O}})^{\mathrm{ord}} := \varprojlim_{r} \mathscr{C}^{i}(\Gamma_{1, r}, W_{\lambda, r})^{\mathrm{ord}}
$$
and similarly for cohomology.
\subsection{Perfection}
\begin{lemma} \label{lem:16}
Let $N$ be a $\Lambda_{r}$-module with trivial $\Gamma_{0,1}$-action. Then
$$
   \mathscr{C}^{\bullet}(\Gamma_{0,1}, \mathbb{D}_{\mathrm{univ}}(\lambda))^{\mathrm{ord}} \otimes_{\Lambda_{r}} N \cong  \mathscr{C}^{\bullet}(\Gamma_{0,1}, \mathbb{D}_{\mathrm{univ}}(\lambda)\otimes_{\Lambda_{r}} N)^{\mathrm{ord}} 
$$
\end{lemma}
\begin{proof}
There's an integer $n$ such that 
$$
     \mathscr{C}^{\bullet}(\Gamma_{0,1}, \mathbb{D}_{\mathrm{univ}}(\lambda)) \cong \oplus_{i = 1}^{n}\mathbb{D}_{\mathrm{univ}}(\lambda)
$$
as $\Lambda_{r}[\Gamma_{0,1}]$-modules. The result follows since the $\Lambda_{r}$-action is continuous and commutes with the Hecke action.
\end{proof}
\begin{lemma}
The $\Lambda_{r}$-regular sequence $([s_{i}^{(r)}] - \chi(s_{i}^{(r)}))_{i}$ generating $I_{r}^{(\chi)}$ is $\mathbb{D}_{\mathrm{univ}}(\lambda)$-regular for any $\lambda$.
\end{lemma}
\begin{proof}
We reduce, using the isomorphism \eqref{eq:12}, to showing that the given sequence is regular in $\Lambda_{0}$, which is visibly the case.
\end{proof}
\begin{proposition} \label{prop:2}
For any choice of $\mathscr{C}_{\bullet}(\Gamma_{0,1})$, the modules $ \mathscr{C}^{\bullet}(\Gamma_{0,1}, \mathbb{D}_{\mathrm{univ}}(\lambda))^{\mathrm{ord}}$ are finite flat $\Lambda_{1}$-modules.
\end{proposition}
\begin{proof}

 
 Let $\mathfrak{m}_{1}$ denote the maximal ideal of $\Lambda_{1}$. There is an exact sequence 
$$
    0 \to \mathfrak{m} \mathscr{C}^{\bullet}(\Gamma_{0,1}, \mathbb{D}_{\mathrm{univ}}(\lambda))^{\mathrm{ord}} \to  \mathscr{C}^{\bullet}(\Gamma_{0,1}, \mathbb{D}_{\mathrm{univ}}(\lambda))^{\mathrm{ord}} \to  \mathscr{C}^{\bullet}(\Gamma_{0,1},\mathbb{D}_{1}(\lambda)/p)^{\mathrm{ord}} \to 0.
$$
In a similar way to Proposition \ref{prop:6} we can show that 
$$
   \mathscr{C}^{\bullet}(\Gamma_{0,1},\mathbb{D}_{1}(\lambda)/p)^{\mathrm{ord}} \cong \mathscr{C}^{\bullet}(\Gamma_{0,1},(\Lambda_{1}/\mathfrak{m}_{1})[L_{r}^{\mathrm{der}}\backslash L_{G}])^{\mathrm{ord}} 
$$
and thus that it is finite.
Thus we can apply the Nakayama lemma to conclude that $ \mathscr{C}^{\bullet}(\Gamma_{0,1}, \mathbb{D}_{\mathrm{univ}}(\lambda))^{\mathrm{ord}}$ is a complex of finitely generated $\Lambda_{1}$-modules. 
The ideal $\mathfrak{m}_{1}$ is generated by a regular sequence $(p, x_{1}, \ldots, x_{n})$, and $\Lambda_{1}/\mathfrak{m}_{1} \cong \mathbb{F}_{p^{k}}$ for $k = [K:\mathbb{Q}_{p}]$. By the local criterion for flatness as stated in \cite[theorem 6.8]{eisenbud}, it suffices to show that 
$$
    \mathrm{Tor}_{1}^{\Lambda_{1}}( \mathscr{C}^{*}(\Gamma_{0,1}, \mathbb{D}_{\mathrm{univ}}(\lambda))^{\mathrm{ord}}, \mathbb{F}_{p^{k}}) = 0.
$$
This group is computed by the Koszul complex for $(p, x_{1}, \ldots, x_{m})$ tensored with $ \mathscr{C}^{*}(\Gamma_{0,1}, \mathbb{D}_{\mathrm{univ}}(\lambda))^{\mathrm{ord}}$, and thus it suffices to prove that the above sequence is $ \mathscr{C}^{*}(\Gamma_{0,1}, \mathbb{D}_{\mathrm{univ}}(\lambda))^{\mathrm{ord}}$-regular. This follows from the fact that this sequence is $\mathbb{D}_{\mathrm{univ}}(\lambda)$-regular and Lemma \ref{lem:16}.
\end{proof}

\begin{corollary}
The complex $\mathscr{C}^{\bullet}(\Gamma_{0,1}, \mathbb{D}_{\mathrm{univ}}(\lambda))^{\mathrm{ord}}$ is a perfect complex of $\Lambda_{0}$-modules concentrated in degrees $[0,\nu]$.
\end{corollary}
\begin{proof}
By the above lemmas the modules $C^{i}$ are finite flat over the local ring $\Lambda_{1}$ and are thus free. We are done by Lemma \ref{lem:11}.
\end{proof}

Write 
$$
M_{\lambda}^{\bullet} : = \mathscr{C}^{\bullet}(\Gamma_{0,1}, \mathbb{D}_{\mathrm{univ}}(\lambda))^{\mathrm{ord}}.
$$
\begin{proof} (of Theorem \ref{thm:5})
By Proposition \ref{prop:2} we have
$$
    M_{\lambda}^{\bullet} \otimes^{L}_{\Lambda_{r}}\mathcal{O}^{(\chi)} \sim \mathscr{C}^{\bullet}(\Gamma_{0,1}, \mathbb{D}_{\mathrm{univ}}(\lambda))^{\mathrm{ord}} \otimes_{\Lambda_{r}} \mathcal{O}^{(\chi)}
$$
by Lemma \ref{lem:4} we have $$
 M_{\lambda}^{\bullet} \otimes^{L}_{\Lambda_{r}}\mathcal{O}^{(\chi)} \sim \mathscr{C}^{\bullet}(\Gamma_{0,1}, \mathbb{D}_{r}(\lambda))^{\mathrm{ord}}
$$
and by Lemma \ref{lem:5} 
$$
     M_{\lambda}^{\bullet} \otimes^{L}_{\Lambda_{r}}\mathcal{O}^{(\chi)} \sim \mathscr{C}^{\bullet}(\Gamma_{1,r}, V_{\lambda,\mathcal{O}})^{\mathrm{ord}}.
$$
\end{proof}

\section{$p$-stabilisation and duality} \label{sec:4}
 
\subsection{$p$-stabilisation}
Let 
$$
    \Gamma = K^{p} \times G(\mathbb{Z}_{p}) \cap G(\mathbb{Q})
$$
for some tame level $K^{p} \subset G(\hat{\mathbb{Z}}^{(p)})$ unramified at almost all primes. 
Define a Hecke operator $\mathcal{T}_{0} = [\Gamma\tau^{-1}\Gamma]$. Let $\mathcal{K} = G(\mathbb{Z}_{p})$ and $J = \tau^{-1} \mathcal{K} \tau \cap \mathcal{K}$. Let $W$ be the Weyl group of $G$, $W_{L}$ the Weyl group of $L_{G}$ and $w_{0}$ the long Weyl element in $W$. For any $W$-set $X$, $x \in X$, $w \in W$ write $x_{w} := w \cdot x$. Let $I$ be the parahoric subgroup associated to $Q_{G}$. 
\begin{lemma}
\begin{equation} \label{eq:8} 
    \mathcal{K}\tau^{-1}\mathcal{K} = \sqcup _{w \in W/W_{L_{G}}} \sqcup_{u \in (w_{0}w)^{-1}Jw_{0}w \cap I\backslash I} uw\tau^{-1}\mathcal{K}
\end{equation}
\end{lemma}
\begin{proof}
The proof given for the $\mathrm{GSp}_{4}$ case in \cite{tilurb} carries over verbatim.
\end{proof}
We will call a weight $\lambda \in X^{\bullet}(T)$  \textit{very regular} if it has trivial stabiliser in the Weyl group. An equivalent formulation is that it does not lie the wall of one of the Weyl chambers. The method of proof of the following theorem is the same as that used in \cite[Proposition 3.2]{tilurb} and \cite[Lemma 7.2]{hidacontrol}. 

\begin{theorem} \label{thm:7}
For $\lambda$ very regular and dominant for $B_{G}$ there is a quasi-isomorphism 
$$
    \mathscr{C}^{\bullet}(\Gamma, V_{\lambda, \mathcal{O}})^{\mathcal{T}_{0}-\mathrm{ord}} \cong  \mathscr{C}^{\bullet}(\Gamma_{0,1}, V_{\lambda, \mathcal{O}})^{\mathrm{ord}} 
$$
given by restricting $e \circ \mathrm{res}$ to the ordinary subspace defined by $\mathcal{T}_{0}$. 
\end{theorem}
\begin{proof}
Consider the restriction map
$$
     \mathrm{res}:   \mathscr{C}^{\bullet}(\Gamma, V_{\lambda, \mathcal{O}}) \to \mathscr{C}^{\bullet}(\Gamma_{0,1}, V_{\lambda, \mathcal{O}}).
$$
and its mod $p^{r}$ reductions
$$
     \mathrm{res}_{r}:  \mathscr{C}^{\bullet}(\Gamma, V_{\lambda, r})\to \mathscr{C}^{\bullet}(\Gamma_{0,1}, V_{\lambda, r}).
$$
Write $\ell, \ell'$ respectively for the restrictions of the maps $e \circ \mathrm{res}$, $e_{0} \circ \mathrm{cores}$ to the ordinary subspaces on their respective sources, and write $\ell_{r}$ for the analagous maps 
$$
    \mathscr{C}^{\bullet}(\Gamma, V_{\lambda, r})^{\mathrm{ord}} \longleftrightarrow \mathscr{C}^{\bullet}(\Gamma_{0,1}, V_{\lambda, r})^{\mathrm{ord}}.
$$
We claim that 
$$\mathcal{T} \circ \mathrm{res}_{1} = \mathrm{res}_{1} \circ \mathcal{T}_{0}
$$
on cohomology.

To prove this it suffices to show that the strata in the decomposition \eqref{eq:8} for $[K\tau^{-1}K]$ corresponding to elements of the Weyl group non-trivial in $W/W_{L}$ when applied to elements of $\mathscr{C}^{i}(\Gamma, V_{\lambda, \mathcal{O}})$ are divisible by $p$, since the stratum corresponding to the trivial element is precisely $V_{0,1}\tau^{-1}V_{0,1}$.

Let $f \in \mathscr{C}^{i}(\Gamma, V_{\lambda, \mathcal{O}})$, then for $z \in \mathscr{C}^{i}(\Gamma)$ there are $v_{i}^{(z)} \in V_{\lambda, \mathcal{O}}$ such that 
$$
    (\mathcal{T}f)(z) = \sum_{u,w}uw\tau^{-1}v_{i}^{(z)}.
$$
Recall that we can consider the $v_{i}^{(z)}$ as functions $G(\mathcal{O}) \to W_{\lambda, \mathcal{O}}$. In this optic we have 
\begin{align*}
    (w\tau^{-1} \cdot v_{i}^{(z)})(g) &= p^{h_{\lambda}}v_{i}^{(z)}(g\tau^{-1}w) \\ 
    &= p^{h_{\lambda}}v_{i}^{(z)}(gw \tau^{-1}_{w}) \\
    &= p^{h_{\lambda} - h_{\lambda_{w^{-1}}}} (\tau^{-1}_{w} \cdot v_{i}^{(z)})(gw).
\end{align*}
We claim that $h_{\lambda} = h_{\lambda_{w^{-1}}}$ if and only if $w \in W_{L}$. We note that if $w \in W_{L}$ then 
\begin{align*}
    h_{\lambda_{w}} &= \langle \eta, \lambda_{w^{-1}} \rangle \\
    &= \langle \eta_{w}, \lambda \rangle \\
    &= \langle \eta, \lambda \rangle \\
    &= h_{\lambda},
\end{align*}
because $\eta$ takes values in $Z(L_{G})$, thus we descend to an action of $W/W_{L_{G}}$ on $\langle \eta, \lambda \rangle$.

Letting $\Phi_{L}$ be the set of simple roots of $L_{G}$ corresponding to $B_{L}$ we note that 
$$
    Z(L_{G}) = \bigcap_{\alpha \in \Phi_{L}} \mathrm{ker}(\alpha),
$$
and thus that $\langle \eta, \alpha \rangle = 0$ for $\alpha \in \Phi_{L}$. Thus for any $w \in W$, if we write $\Phi_{L}' = \Phi_{G} \backslash \Phi_{L}$ for the set of relative roots of $L_{G}$, then as $\lambda$ is dominant there are non-negative integers $n_{\alpha}$ such that 
$$
    \lambda - \lambda_{w} = \sum_{\alpha \in \Phi^{+}_{L}} n_{\alpha}\alpha + \sum_{\alpha \in (\Phi_{L}^{+})'} n_{\alpha}\alpha
$$
whence it is clear that 
$$
   h_{\lambda} \geq  h_{\lambda_{w}}
$$
with equality holding in either of two cases:
\begin{enumerate}
    \item $\lambda = \lambda_{w}$
    \item $\lambda - \lambda_{w} = \sum_{\alpha \in \Phi_{L}}n_{\alpha}\alpha$ where at least one $n_{\alpha} > 0$.
\end{enumerate}
The first case directly contradicts the very regularity of $\lambda$. In the second case we want to show that $\lambda_{w}$ being in the cone 
$$
    C = \{\lambda - \sum_{\alpha \in \Phi_{L}}n_{\alpha}\alpha: n_{\alpha} \geq 0\}
$$
implies that $w \in W_{L}$. We note that if $\lambda$ is very regular then any of its conjugates under $W$ are also regular as $W$ conjugates the interiors of the Weyl chambers. Let $\mu$ be any regular weight and note also that a simple reflection $s_{\alpha}$ corresponding to $\alpha \in \Phi_{G}$ stabilises the cone $\{\mu - n_{\alpha}\alpha: n_{\alpha} \geq 0\}$. By very regularity there is $n > 0$ such that $s_{\alpha} \cdot \mu = \mu - n\alpha$. Expanding an element of $W$ as a string of such reflections 
$$
    w = s_{\alpha_{m}}\ldots s_{\alpha_{1}},
$$
we see that $\lambda_{w} \in C$ implies that no reflection appearing in the expansion of $w$ can be in $\Phi'_{L}$, thus $w \in W_{L}$. We conclude that
$$
    w \notin W_{L} \implies uw\tau^{-1}v_{i}^{(z)} \equiv 0 \ \mathrm{mod} \ p
$$
and thus $\mathcal{T}_{0} \equiv \mathcal{T} \ \mathrm{mod} \ p$. 

In particular we have 
$$
\mathcal{T} \circ \mathrm{res}_{r} = \mathrm{res}_{1} \circ \mathcal{T}_{0},
$$
which implies that $\mathrm{res}$ maps between the ordinary subspaces for the respective Hecke operators.  
Recall the definition of the ordinary idempotents 
$$
    e_{?} = \lim_{n \to \infty}\mathcal{T}_{?}^{n!}
$$
for $? \in \{\emptyset, 0\}$.

The above result shows that 
$$
    e \circ \mathrm{res}_{1} = \mathrm{res}_{1} \circ e_{0}.
$$
To show the map $\ell_{1}'$ is surjective we note that 
$$
   \ell_{1}' \circ\ell_{1} = [\Gamma: \Gamma_{0,1}].
$$
This index is prime to $p$ so injectivity follows.

Proving surjectivity for $\ell_{1}$ is a little trickier. Define maps 
\begin{align*}
    \mathscr{P}: H^{i}(\Gamma_{0,1}, V_{\lambda, 1}) &\to H^{i}(\Gamma_{0,1}, W_{\lambda, 1}) \\
    \iota: H^{i}(\Gamma_{0,1}, W_{\lambda_{w_{0}}, 1}) &\to  H^{i}(\Gamma_{0,1}, V_{\lambda, 1})
\end{align*}
via the natural maps 
\begin{align*}
    V_{\lambda, 1} &\to W_{\lambda, 1} \\
    W_{\lambda, 1} &\to V_{\lambda, 1}
\end{align*}
given by evaluation at $1$ and inclusion respectively. These maps are intertwined by the map
$$
    [\mathcal{W}]: H^{i}(\Gamma_{0,1}, W_{\lambda, 1}) \to H^{i}(\Gamma_{0,1}, W_{\lambda_{w_{0}}, 1})
$$
induced by $\mathcal{W} = w_{0}\tau^{-1} \in \Gamma\tau^{-1}$. Recall that $\mathscr{P}$ restricts to an isomorphism on $\mathcal{T}$-ordinary subspaces.  We will show that the map $$
    \mathscr{P}\circ \mathrm{res}^{\Gamma}_{\Gamma_{0,1}} \circ \mathrm{cores}^{\Gamma}_{\Gamma_{0,1}} \circ \iota \circ [W]:  H^{i}(\Gamma_{0,1}, W_{\lambda, 1}) \to  H^{i}(\Gamma_{0,1}, W_{\lambda, 1})
$$
is a bijection on $\mathcal{T}$-ordinary subspaces, whence we can conclude that the restriction map is surjective. Note that 
\begin{align*}
    \Gamma &= \sqcup_{u,w}uw(\tau\Gamma \tau^{-1} \cap \Gamma) \\
    &= \sqcup_{u,w}uw(w_{0}\Gamma_{0,1}w_{0}^{-1}) \\
    &= \sqcup_{u,w}uww_{0}\Gamma_{0,1},
\end{align*}
where $u,w$ are as in \eqref{eq:8}, and the last line is obtained by multiplying both sides on the right by $w_{0}$ which can be assumed to be in $\Gamma$ by weak approximation. Let $f \in \mathscr{C}^{i}(\Gamma_{0,1}, W_{\lambda, 1})$, then for $z \in \mathscr{C}_{i}(\Gamma_{0,1})$ we have
$$
(\mathrm{res} \circ \mathrm{cores} \circ \iota)(\mathcal{W}\cdot f)(z) = \sum_{u,w}uww_{0}(\mathcal{W}\cdot f)(w_{0}^{-1}w^{-1}u z).
$$
Since $(\mathcal{W}\cdot f)$ takes values in $W_{\lambda_{w_{0}, 1}}$ we see that only the terms for $w = 1$ will be non-zero under $\mathscr{P}$, so 
$$
    (\mathscr{P} \circ \mathrm{res} \circ \mathrm{cores} \circ \iota)(\mathcal{W}\cdot f)(z) = \sum_{u}uw_{0}(\mathcal{W}\cdot f)(w_{0}^{-1}w^{-1}u z)
$$
but  
$$
    uw_{0}(\mathcal{W}\cdot f)(w_{0}^{-1}w^{-1}u z) = u\tau^{-1}f(\tau u^{-1} z)
$$
so the sum is just the expression for the Hecke operator $\mathcal{T}$ and this is invertible on the $\mathcal{T}$-ordinary subspace by definition.
 
Now that we know that the maps 
$$
   \ell_{1}, \ell_{1}'
$$
are both surjective, we see from the diagram
$$
\begin{tikzcd}
 & H^{i}(\Gamma, V_{\lambda,r})^{\mathrm{ord}} \arrow[dl, ""] \arrow[r, "\ell_{r}"'] \arrow[d, ""] & H^{i}(\Gamma_{0,1}, V_{\lambda,r})^{\mathrm{ord}} \arrow[l, bend right, "\ell_{r}'"] \arrow[d,""] \arrow[dr] & \\
 H^{i}(\Gamma, V_{\lambda,r})^{\mathrm{ord}}/p \arrow[r, hook]  & H^{i}(\Gamma, V_{\lambda,1})^{\mathrm{ord}}  \arrow[r, "\ell_{1}"'] & H^{i}(\Gamma_{0,1}, V_{\lambda,1})^{\mathrm{ord}} \arrow[l, bend right, "\ell_{1}'"] &  H^{i}(\Gamma_{0,1}, V_{\lambda,r})^{\mathrm{ord}}/p \arrow[l, hook]
\end{tikzcd}
$$
that we can use Nakayama's lemma to deduce that the maps $\ell_{r}, \ell_{r}'$ are also both surjective. Since these are finite sets we can deduce that $\ell_{r}$ is a bijection for all $r$ and by taking the inverse limit we get that $\ell$ is an isomorphism.
 Therefore we get the desired quasi-isomorphism
$$
    \mathscr{C}^{\bullet}(\Gamma, V_{\lambda, \mathcal{O}})^{\mathcal{T}_{0}-\mathrm{ord}} \sim \mathscr{C}^{\bullet}(\Gamma_{0,1}, V_{\lambda, \mathcal{O}})^{\mathrm{ord}}.
$$
\end{proof}
 
Let $F$ be a number field over which $G$ splits. By various results of Borel and Tits there is a finite set of primes $S$ such that $G$ has a split reductive model over the ring of $S$-integers  $\mathcal{O}_{F, S}$. We note that this implies that for every $p \notin S$ the group $G$ is unramified at $p$. 
For $\lambda \in X^{\bullet}(T_{G})$ dominant for $B_{L_{G}}$ and $\chi \in X^{\bullet}(L_{G})$ such that $\lambda + \chi$ is dominant for $G$, define 
\begin{align*}
    W_{\lambda, F,S} &=  \{f \in \mathcal{O}_{F,S}[L_{G}]: f(\bar{b}x) = \lambda(\bar{b})f(x) \ \forall \ \bar{b} \in \bar{B}_{L_{G}}\}\\
    V_{\lambda, F, S}  &= \{f \in \mathcal{O}_{F,S}[G] \otimes W_{\lambda, F,S }: f(\bar{q}g) = \chi(\bar{q})\bar{q}f(g) \ \forall \ \bar{q} \in \bar{Q}_{G}\}.
\end{align*}
Note that these $\mathcal{O}_{F,S}$ modules are independent of $p$. It is clear from our previous definition of the action of $\tau^{-1}$ that it preserves these modules and thus we have an action of $\mathcal{T}_{0}$ on the cohomology groups $H^{i}(\Gamma,  V_{\lambda, F, S})$.  For a prime $p$ write $F_{p}$ for a $p$-adic completion of $F$ and $\mathcal{O}_{F_{p}}$ for its ring of integers. 
\begin{corollary} \label{cor:4}
There is a finite set of primes $S_{\Gamma}$ containing $S$ such that for primes $p \notin S_{\Gamma}$ the cohomology groups $H^{i}(\Gamma_{0,1}, V_{\lambda, \mathcal{O}_{F_{p}}})^{\mathrm{ord}}$ are torsion-free as $\mathcal{O}$-modules.
\end{corollary}
\begin{proof}
The cohomology group $e_{0}H^{i}(\Gamma,  V_{\lambda, F, S})$ give an $\mathcal{O}_{F,S}$-lattice in $e_{0}H^{i}(\Gamma,  V_{\lambda, \mathcal{O}_{F_{p}}})$ and thus, by the previous theorem, in $eH^{i}(\Gamma_{0,1}, V_{\lambda, \mathcal{O}_{F_{p}}})$. Since this lattice is independent of $p$ and finitely generated if we let 
$$
S_{\Gamma} = S \cup \{ \text{torsion primes in 
$e_{0}H^{i}(\Gamma,  V_{\lambda, F, S})$}\}
$$
then for all $p \notin S_{\Gamma}$ the $\mathcal{O}_{F_{p}}$-module $eH^{i}(\Gamma_{0,1}, V_{\lambda, \mathcal{O}_{F_{p}}}) = e_{0}H^{i}(\Gamma,  V_{\lambda, F, S}) \otimes \mathcal{O}_{F_{p}}$ is torsion free. 
\end{proof}
\subsection{Duality results} \label{subsec:1}

For this section only set $Q_{G} = B_{G}$. We show how to obtain a derived control result for compactly supported group cohomology. We show that the $\mathcal{O}$-linear Poincar\'e duality pairing at level $\Gamma_{0,1}$ is non-degenerate outside of a finite set of primes. Throughout this section $\lambda$ is dominant for $B_{G}$. 

\begin{definition}
For any left $\Gamma_{0,1}$-module $M$ define a complex
$$
   \mathscr{C}_{c}^{\bullet}(\Gamma_{0,1}, M) = \mathscr{C}_{2d - \bullet}(\Gamma) \otimes_{\Gamma_{0,1}} M,
$$
where $2d := \mathrm{dim}_{\mathbb{R}}G(\mathbb{R})/K_{\infty}$ for a maximal open compact subgroup $K_{\infty} \subset G(\mathbb{R})$.
\end{definition}
This complex satisfies 
$$
      H^{i}(\mathscr{C}_{c}^{\bullet}(\Gamma_{0,1}, M)) = H^{i}_{c}(\Gamma_{0,1}, M)
$$

We can define a Hecke operator $\mathcal{T}$ on $\mathscr{C}_{c}^{\bullet}(\Gamma_{0,1}, M)$ in the same way as for $\mathscr{C}^{\bullet}(\Gamma_{0,1}, M)$ and whence define the ordinary part of the complex, denoted $\mathscr{C}_{c}^{\bullet}(\Gamma_{0,1}, M)^{\mathrm{ord}}$ and uniquely defined up to homotopy equivalence. 

\begin{remark}
As modules we have 
$$
      \mathscr{C}_{c}^{i}(\Gamma_{0,1}, M) =   \mathscr{C}^{i}(\Gamma_{0,1}, M).
$$
In particular $ \mathscr{C}_{c}^{\bullet}(\Gamma_{0,1}, \mathbb{D}_{\mathrm{univ}}(\lambda))^{\mathrm{ord}}$  is a perfect complex of $\Lambda_{0}$-modules.
\end{remark}

\begin{definition}
We define a pairing of $\mathcal{O}$-modules
$$
    ( -, -)_{r}: H_{c}^{i}(\Gamma_{1,r}, (V_{\lambda, \mathcal{O}})^{\vee}) \otimes_{\mathcal{O}} H^{2d - i}(\Gamma_{1,r}, V_{\lambda, \mathcal{O}}) \to \mathcal{O}.
$$
by 
$$
    (x, y)_{r} = \varphi_{M}(x \cup \mathcal{W}(y))
$$
where $\varphi_{\lambda}$ is the natural map $V_{\lambda, \mathcal{O}} \otimes V_{\lambda, \mathcal{O}} \to \mathcal{O}$.
\end{definition}

We note that $(V_{\lambda, \mathcal{O}})^{\vee} \neq V_{\lambda, \mathcal{O}}$; we can at most say that
$$
(V_{\lambda, \mathcal{O}})^{\vee} \subset V_{\lambda, \mathcal{O}}.
$$

\begin{lemma} \label{lem:25}
Let $V_{\lambda, \mathcal{O}}^{\mathrm{min}}$ be the minimal admissible lattice in $V_{\lambda}$. There is an isomorphism 
$$
    \mathscr{C}^{i}(\Gamma_{0,1}, V_{\lambda, \mathcal{O}}^{\mathrm{min}})^{\mathrm{ord}} \cong  \mathscr{C}^{i}(\Gamma_{0,1}, V_{\lambda, \mathcal{O}})^{\mathrm{ord}}
$$
\end{lemma}
\begin{proof}
  The strict dominance requirement for $\eta$ forces 
\begin{equation} \label{val:1}
    h_{\lambda} - h_{\mu}> 0
\end{equation}
for all characters $\mu \neq \lambda$ of $T_{G}$ appearing in $V_{\lambda}$. We recall that, given a choice of highest weight vector $v_{\lambda}$, an admissible lattice in $V_{\lambda}$ is the direct sum of its intersections with the weight spaces for $T_{G}$, and the highest weight spaces are the $\mathcal{O}$-linear span of $v_{\lambda}$. By construction of the $*$ action, $\tau^{-1}$ fixes $\lambda$ and for $\mu \neq \lambda$ we have that $\tau^{-r}*(V_{\lambda,\mathcal{O}})_{\mu} = p^{r(h_{\lambda} - h_{\mu})}(V_{\lambda,\mathcal{O}})_{\mu}$. By \eqref{val:1} we see that 
$$
    \lim_{r} \tau^{-r} * V_{\lambda, \mathcal{O}} = \mathcal{O} \cdot v_{\lambda}.
$$
In particular, as $V_{\lambda, \mathcal{O}}^{\mathrm{min}}$ is an open subgroup of $V_{\lambda, \mathcal{O}}$ we see that for $r  \gg 0$: 
$$
    \tau^{-r} * V_{\lambda, \mathcal{O}} \subset V_{\lambda, \mathcal{O}}^{\mathrm{min}}.
$$
It's easy to see that this implies the result.
\end{proof}

We now discuss the Hecke action on $(-,-)_{r}$. We define an `adjoint' Hecke operator $\mathcal{T}^{*}$ defined by 
$$
    \mathrm{cores}_{\Gamma' \cap \tau^{-1}\Gamma\tau}^{\Gamma'} \circ [\tau] \circ \mathrm{res}^{\Gamma}_{\tau\Gamma'\tau^{-1} \cap \Gamma}
$$
and with action on $V_{\lambda}^{\vee}$ given by $\tau * v = p^{-h_{\lambda}}\tau v$. 
This satisfies 
$$
    \mathcal{T} x \cup y = x \cup  \mathcal{T}^{*}y.
$$
\begin{lemma}
The ordinary idempotent $e_{\mathrm{ord}}$ is self-adjoint for the pairing $(-,-)_{r}$.
\end{lemma}
\begin{proof}
Note that $\mathcal{W}$ normalizes $V_{r}$. Thus
\begin{align*}
    \mathcal{W}^{-1}\Gamma \tau \Gamma \mathcal{W} &= \tau w_{0}^{-1} \Gamma \tau \Gamma w_{0} \tau^{-1} \\
    &=  \Gamma \tau w_{0}^{-1} \tau  w_{0} \tau^{-1} \Gamma \\
    &= \Gamma w_{0}^{-1}\tau w_{0} \Gamma.
\end{align*}
The element $w_{0}^{-1}\tau w_{0} \in T(\mathbb{Q}_{p})$ is the image of $p$ under the cocharacter $\eta^{-1}_{w_{0}}$ which satsfies $\langle \eta^{-1}_{w_{0}}, \alpha \rangle > 0$ for all positive roots $\alpha$. Thus $\Gamma w_{0}^{-1}\tau w_{0} \Gamma$ defines the same idempotent $e_{\mathrm{ord}}$ as $\Gamma \tau^{-1} \Gamma$ so we are done. 
\end{proof}

\begin{proposition}
There is a cofinite set of primes for which the pairing 
$$
    (-,-)_{r}:   H^{i}_{c}(\Gamma_{1,r}, V_{\lambda, \mathcal{O}})^{\mathrm{ord}} \otimes   H^{2d - i}(\Gamma_{1,r}, V_{\lambda, \mathcal{O}})^{\mathrm{ord}} \to \mathcal{O}
$$
is non-degenerate. 
\end{proposition}
\begin{proof}
Note that $\mathcal{W}$ induces an isomorphism
$$
     H^{j}(\Gamma_{1,r}, V_{\lambda, \mathcal{O}})^{\mathrm{ord}} \cong  H^{j}(\Gamma_{1,r}, V_{\lambda, \mathcal{O}})^{\mathcal{T}^{*}-\mathrm{ord}}
$$
where the latter is the ordinary subspace for the adjoint Hecke operator $\mathcal{T}^{*}$. Thus non-degeneracy of the pairing is implied by non-degeneracy of the standard Poincar\'e duality pairing because this restricts to a pairing between the ordinary and $\mathcal{T}^{*}$-ordinary subspaces. This pairing is well-known to descend to a non-degenerate pairing 
$$
    H^{i}_{c}(\Gamma_{1,r}, V_{\lambda, \mathcal{O}})/(\mathrm{tors})\otimes   H^{2d - i}(\Gamma_{1,r}, V_{\lambda, \mathcal{O}})/(\mathrm{tors}) \to \mathcal{O}.
$$
We know from Corollary \ref{cor:4} (and the analagous version for compactly supported cohomology) that the groups $H^{j}_{c}(\Gamma_{1,r}, V_{\lambda, \mathcal{O}})^{\mathrm{ord}}, H^{j}(\Gamma_{1,r}, V_{\lambda, \mathcal{O}})^{\mathrm{ord}}$ are free $\mathcal{O}$-modules for a cofinite set of primes and thus the result follows.
\end{proof}
Thus the pairing $(-,-)_{r}$ descends to a pairing 
$$
        H^{i}_{c}(\Gamma_{1,r}, V_{\lambda, \mathcal{O}})^{\mathrm{ord}} \otimes   H^{2d - i}(\Gamma_{1,r}, V_{\lambda, \mathcal{O}})^{\mathrm{ord}} \to \mathcal{O}
$$
which is non-degenerate outside of a finite set of primes $S_{\Gamma}$.

\begin{remark}
Although one can extend the proof of Lemma \ref{lem:25} to the situation that $\lambda$ is a character of $L_{G}$ for a general parabolic $Q_{G}$, a generalisation of the duality pairing to general parabolics is hampered by the fact that the `Atkin-Lehner' style operator $\mathcal{W}$ does not necessarily preserve Levi subgroups.
\end{remark}
\section{Adèlic cohomology and Hecke algebras} \label{sec:5}
\subsection{Adèlic cohmology and localisations}

For $U =U_{?,r} \subset G(\mathbb{A}_{f})$, let $\mathbb{T}(U)$ be the Hecke algebra of locally constant $U$-biinvariant functions on $G(\mathbb{A}_{f})$. It is generated by indicator functions on double cosets $U\alpha U$, $\alpha \in G(\mathbb{A}_{f})$. We write $\mathbb{T}^{S}(U)$ for the Hecke algebra restricted to places at which $U$ is hyperspecial and let
$$
\mathbb{T}^{S}_{p}(U) = \langle \mathbb{T}^{S}(U), \mathcal{T} \rangle \subset \mathbb{T}(U).
$$
This subalgebra is well known to be commutative. We will often drop the open compact $U$ from the notation. Assuming $G^{\mathrm{der}}$ sastisfies strong approximation, there is a finite set $\{t_{1}, \ldots, t_{n}\}$ such that 
$$
    G(\mathbb{A}_{f}) = \sqcup_{i} G(\mathbb{Q}) \times t_{i} U,
$$
and if we define 
$$
    \Gamma(t_{i}, U) := t_{i}Ut_{i}^{-1} \cap G(\mathbb{Q}) 
$$
then for any $\mathcal{O}[U]$-module $M$ with a compatible $\tau^{-1}$ action, $\mathbb{T}_{p}^{S}(U)$ acts naturally on the complex
$$
    R\Gamma(U, M) := \bigoplus_{i} \mathscr{C}^{\bullet}(\Gamma(t_{i}, U), M).
$$
It can be shown that the image of this complex in the homotopy category does not depend on the choice of $t_{i}$. We can define the double coset action (and thus the action of $\mathbb{T}_{p}^{S}(U)$) in terms of non-adèlicised Hecke operators (defined in Section \ref{sec:1b}) on summands in the following way: 

 Let $x \in G(\mathbb{A}_{f})$ have $p$-part $\tau^{-1}$ and let  $xt_{i} = \gamma_{x,i}t_{j_{i}} k \in G(\mathbb{Q})t_{j_{i}}U$. The action is then given by:
$$
 [UxU] = \bigoplus_{i} [\Gamma(t_{i}, U) \gamma_{x, i}^{-1} \Gamma(t_{j_{i}}, U)].
$$
As $\mathbb{T}_{p}^{S}(U)$ is generated by such operators, this description suffices to define the action of $\mathbb{T}_{p}^{S}(U)$.

The definition of these coset operators as given by \eqref{eq:13} in terms of functorial morphisms shows that the double cosets $[V_{1,r}xV_{1,r}]$ for $r \geq 1$ commute with corestriction maps and thus define Hecke operators on the inverse limit cohomology 
$$
    R\Gamma(V_{1,\infty}, M) := \varprojlim_{r}R\Gamma(V_{1,r}, M).
$$
In particular we get an action of $\mathcal{T}$ and can use this to define the ordinary complex $ R\Gamma(V_{1,\infty}, M)^{\mathrm{ord}}$ defined uniquely up to homotopy and satisfying
$$
  R\Gamma(V_{1,\infty}, M)^{\mathrm{ord}} =  \varprojlim_{r}R\Gamma(V_{1,r}, M)^{\mathrm{ord}} 
$$
in the homotopy category.

Since by weak approximation we can choose $t_{i} \in G(\hat{\mathbb{Z}})$ to be trivial at $p$, by looking on individual summands we obtain derived control theorems for this `adèlic cohomology'. 

For $? \in \{\emptyset, p\}$, write $\mathbb{T}^{S}_{?}(U, M)$ for the image of $\mathbb{T}^{S}_{?} \otimes R$ in $\mathrm{End}(R\Gamma(U, M))$.
Write $\mathbb{T}_{\lambda}^{S}:= \mathbb{T}^{S}(V_{0,1}, \mathbb{D}_{\mathrm{univ}})^{\mathrm{ord}}$, where the superscript ord refers to restricting the ordinary subspace. This is well-defined as $\mathbb{T}^{S}_{p}$ is commutative and acts continuously. We prove a variant of the derived control theorem for localisations by ideals of $\mathbb{T}^{S}_{\lambda}$. The proof of the following theorem is pretty much identical to \cite[Theorem 5.1]{ashpol}.

\begin{theorem}
There is a natural bijection (not an isomorphism of schemes)
$$
    \mathrm{Spec} \ \mathbb{T}_{\lambda}^{S}/I^{(\chi)}_{r} \leftrightarrow \mathrm{Spec} \ \mathbb{T}^{S}(V_{1,r}, V_{\lambda + \chi, \mathcal{O}})^{\mathrm{ord}}.
$$
\end{theorem}
\begin{proof}
We note that since the elements of $\mathbb{T}^{S}_{\lambda}$ are $\Lambda_{0}$-linear there is a natural map 
$$
    q: \mathbb{T}^{S}_{\lambda} \to \mathbb{T}^{S}(V_{1,r}, V_{\lambda + \chi, \mathcal{O}})^{\mathrm{ord}}
$$
which lifts into a diagram
$$
\begin{tikzcd}
\mathbb{T}^{S} \arrow[d, ""] \arrow [dr, ""] \\
\mathbb{T}^{S}_{\lambda} \arrow[r, "q"] & \mathbb{T}^{S}(V_{1,r}, V_{\lambda + \chi, \mathcal{O}})^{\mathrm{ord}}.
\end{tikzcd}
$$
The two vertical maps are surjective and thus so is $q$.

To prove the theorem it suffices to show that $\mathrm{ker}(q) \subset \mathrm{Rad}(I_{r}^{(\chi)}\mathbb{T}^{S}_{\lambda})$. The ideal $I_{r}^{(\chi)} \subset \Lambda_{r}$ is generated by a regular sequence $(x_{1}, \ldots, x_{m})$.  We proceed by induction on $m$. Suppose $I^{(\chi)}_{r} = (x)$ and let $T \in \mathrm{ker}(q)$. Then $T(\mathscr{C}^{\bullet}(V_{0,1}, \mathbb{D}_{\mathrm{univ}}(\lambda))^{\mathrm{ord}}) \subset x\mathscr{C}^{\bullet}(V_{0,1}, \mathbb{D}_{\mathrm{univ}}(\lambda))^{\mathrm{ord}}$. We know that the modules $\mathscr{C}^{i}(V_{0,1}, \mathbb{D}_{\mathrm{univ}}(\lambda))^{\mathrm{ord}}$ have no $x$-torsion, so 
$$
    T' := x^{-1}T \in \mathrm{End}_{\Lambda_{0}}(M_{\lambda}^{\bullet})
$$
is well-defined. Define 
$$
    \mathscr{X} = \{\phi \in \mathrm{End}_{\Lambda_{0}}(M_{\lambda}^{\bullet}): \exists k \ \text{such that} \ x^{k}\phi \in \mathbb{T}^{S}_{\lambda}\}.
$$
Then $\mathscr{X}$ is a finitely generated $\Lambda_{0}$-module and $T' \in \mathscr{X}$. Since $\mathscr{X}$ is finitely generated there is an integer $N$ such that $x^{N}\mathscr{X} \subset \mathbb{T}^{S}_{\lambda}$. Thus
$$
    T^{N + 1} = x(x^{N}T^{N + 1}) \in \mathbb{T}^{S}_{\lambda},
$$
so $T \in \mathrm{Rad}(I_{r}^{(\chi)}\mathbb{T}^{S}_{\lambda})$. 

Now let $m \geq 1$ and suppose the induction hypothesis holds for regular sequences of length $m$. Suppose $I_{r}^{(\chi)} = (x_{1}, \ldots, x_{m + 1})$ then consider 
$$
    \psi: \mathbb{T}^{S}_{\lambda} \xrightarrow{\varphi} \mathbb{T}^{S}(M^{\bullet}_{\lambda}/(x_{1}, \ldots, x_{m})) \to \mathbb{T}^{S}(V_{1,r}, V_{\lambda + \chi, \mathcal{O}})^{\mathrm{ord}}.
$$
We note that the proof of the $m = 1$ case holds perfectly well for the second map in the above sequence and so if $T \in \mathrm{ker}(\psi)$ then $\varphi(T) \in \mathrm{Rad}(x_{m + 1}\mathbb{T}^{S}(M^{\bullet}_{\lambda}/(x_{1}, \ldots, x_{m})))$ so that there is $N$ such that $\varphi(T^{N})$ is in this radical and thus there is an element $z \in \mathbb{T}^{S}_{\lambda}$ such that $\varphi(T^{N}) = x_{m + 1}\varphi(z)$. By the induction hypothesis
$$
    T^{N} - \alpha z \in \mathrm{Rad}((x_{1}, \ldots, x_{m})\mathbb{T}^{S}_{\lambda})
$$
so there is $M$ such that 
$$
     (T^{N} - \alpha z)^{M} \in (x_{1}, \ldots, x_{m})\mathbb{T}^{S}_{\lambda}
$$
whence it is clear that 
$$
    T^{NM} \in I^{(\chi)}_{r}\mathbb{T}^{S}_{\lambda}. 
$$
\end{proof}

\begin{corollary} \label{cor:5}
Let $\wp 
\in \mathrm{Spec} \ \mathbb{T}^{S}_{\lambda}/I^{(\chi)}_{r}$, then 
$$
    (M_{\lambda}^{\bullet})_{\wp} \otimes^{L}_{\Lambda_{r}} \mathcal{O}^{(\chi)} \sim R\Gamma(V_{1,r}, V_{\lambda + \chi, \mathcal{O}})_{\wp}^{\mathrm{ord}}
$$
\end{corollary}
\begin{proof}
The control theorem tells us precisely how prime ideals move between the big and small Hecke algebras and localisation is exact. 
\end{proof}
Let $V = K^{p} \times G(\mathbb{Z}_{p})$. 
\begin{theorem}
Suppose there is a maximal ideal $\mathfrak{m} \in \mathrm{Spec}\mathbb{T}^{S}$ such that 
$$
    H^{\bullet}_{?}(V_{0,1}, V_{\lambda, \mathcal{O}})_{\mathfrak{m}} =   H^{d}_{?}(V_{0,1}, V_{\lambda, \mathcal{O}})_{\mathfrak{m}}
$$
for $? \in \{\emptyset, c\}$
and $H^{d}_{?}(V_{0,1}, V_{\lambda, \mathcal{O}})_{\mathfrak{m}}$ is free as an $\mathcal{O}$-module. Let $\Delta^{p}$ be the prime-to-$p$ part of $S_{G}(\mathbb{Z}_{p})$ Then 
$$
    H^{\bullet}_{?}(V_{1,\infty}, V_{\lambda, \mathcal{O}})_{\mathfrak{m}}^{\mathrm{ord}, \Delta^{p}} =   H^{d}_{?}(V_{1,\infty}, V_{\lambda, \mathcal{O}})_{\mathfrak{m}}^{\mathrm{ord}, \Delta^{p}},
$$
where $(\cdot)^{\Delta^{p}}$ refers to $\Delta^{p}$-invariants. Furthermore, $H^{d}_{?}(V_{1,\infty}, V_{\lambda, \mathcal{O}})_{\mathfrak{m}}^{\mathrm{ord}, \Delta^{p}}$ is a free $\Lambda_{1}$-module.
\end{theorem}

\begin{proof}
 We note that as $\Delta^{p}$ has prime-to-$p$ order, taking $\Delta^{p}$ invariants is an exact functor on $\mathcal{O}$-modules. Let $\mathfrak{m}_{r}$ be the maximal ideal of $\Lambda_{r}$ generated by $p$ and $I_{r}^{(\mathbbm{1})}$. 
 An easy corollary of derived control is the following `mod $p$' variant:
$$
    M_{\lambda,?}^{\bullet} \otimes^{\mathbb{L}}_{\Lambda_{r}} \Lambda_{r}/\mathfrak{m}^{}_{r} \sim R\Gamma_{?}(V_{1,r}, V_{\lambda, 1})^{\mathrm{ord}}.
$$
 We note that since $H_{?}^{i}(V_{0,1}, V_{\lambda, \mathcal{O}})_{\mathfrak{m}}^{\mathrm{ord}}$ is a free $\mathcal{O}$-module for all $i$, the short exact sequence
$$
    0 \to V_{\lambda,\mathcal{O}} \xrightarrow{p}  V_{\lambda,\mathcal{O}} \to  V_{\lambda,1} \to 0
$$
gives us that $H_{?}^{i}(V_{0,1}, V_{\lambda, \mathcal{O}})_{\mathfrak{m}}^{\mathrm{ord}}/p = H_{?}^{i}(V_{0,1}, V_{\lambda, 1})_{\mathfrak{m}}^{\mathrm{ord}}$ and so 
$$
     H_{?}^{i}(V_{0,1}, V_{\lambda, 1})_{\mathfrak{m}}^{\mathrm{ord}} = 0
$$
for all $i$.
We consider the spectral sequence 
 $$
  E_{2}^{i,j}: \mathrm{Tor}_{-i}^{\Lambda_{1}}(H^{j}_{?}(V_{1,\infty}, V_{\lambda, \mathcal{O}})^{\mathrm{ord}}, \Lambda_{1}/\mathfrak{m}_{1}) \implies H_{?}^{i + j}(U_{1,1}, V_{\lambda, 1})^{\mathrm{ord}},
 $$
 which satisfies $E^{i,j}_{2} = 0$ for $j > \nu$. We can apply the exact functor $(\cdot)^{\Delta^{p}}$ to the spectral sequence to get 
 $$
      (E_{2}^{i,j})^{\Delta^{p}}: \mathrm{Tor}_{-i}^{\Lambda}(H^{j}_{?}(V_{1,\infty}, V_{\lambda, \mathcal{O}})^{\mathrm{ord}}, \Lambda/\mathfrak{m}_{1})^{\Delta^{p}}\implies H_{?}^{i + j}(U_{1,1}, V_{\lambda, 1})^{\mathrm{ord}, \Delta^{p}} =  H_{?}^{i + j}(V_{0,1}, V_{\lambda, 1})^{\mathrm{ord}}.
 $$
 Moreover, we can show that 
$$
 \mathrm{Tor}_{-i}^{\Lambda_{1}}(H^{j}_{?}(V_{1,\infty}, V_{\lambda, \mathcal{O}})^{\mathrm{ord}}, \Lambda_{1}/\mathfrak{m}_{1})^{\Delta^{p}} = \mathrm{Tor}_{-i}^{\Lambda_{1}}(H^{j}_{?}(V_{1,\infty}, V_{\lambda, \mathcal{O}})^{\mathrm{ord}, \Delta^{p}}, \Lambda_{1}/\mathfrak{m}_{1}),
$$
(see Lemma \ref{lem:26}) so the spectral sequence becomes
 $$
      (E_{2}^{i,j})^{\Delta^{p}}: \mathrm{Tor}_{-i}^{\Lambda_{1}}(H^{j}_{?}(V_{1,\infty}, V_{\lambda, \mathcal{O}})^{\mathrm{ord}, \Delta^{p}}, \Lambda_{1}/\mathfrak{m}_{1})\implies H_{?}^{i + j}(V_{0,1}, V_{\lambda, 1})^{\mathrm{ord}}.
 $$
 We can then read off that 
$$
     H_{?}^{\nu}(V_{1,\infty}, V_{\lambda, \mathcal{O}})_{\mathfrak{m}}^{\mathrm{ord}, \Delta^{p}}/\mathfrak{m}_{1} \cong H_{?}^{\nu}(V_{0,1}, V_{\lambda, 1})_{\mathfrak{m}}^{\mathrm{ord}} = 0
$$  
and thus Nakayama allows us to conclude that
$$
    H_{?}^{\nu}(V_{1,\infty}, V_{\lambda, \mathcal{O}})_{\mathfrak{m}}^{\mathrm{ord}, \Delta^{p}} = 0.
$$
Now $(E_{2}^{i,j})^{\Delta^{p}} = 0$ for $j > \nu - 1$ and so we perform this step inductively to get $H_{?}^{j}(V_{1,\infty}, V_{\lambda, \mathcal{O}})_{\mathfrak{m}}^{\mathrm{ord},\Delta^{p}} = 0$ for $j > d$, so that $(E^{i,j}_{2})^{\Delta^{p}} = 0$ for $j > d$. We can now read off that there is a surjection
$$
   H_{?}^{d - 1}(V_{0, 1}, V_{\lambda, 1})_{\mathfrak{m}}^{\mathrm{ord},\Delta^{p}} =  H_{?}^{d - 1}(V_{0,1}, V_{\lambda, 1})_{\mathfrak{m}}^{\mathrm{ord}} \twoheadrightarrow \mathrm{Tor}_{1}^{\Lambda}(
     H_{?}^{d}(V_{1,\infty}, V_{\lambda, \mathcal{O}})_{\mathfrak{m}}^{\mathrm{ord}, \Delta^{p}}, \Lambda/\mathfrak{m}_{1}),
$$
and since $ H_{?}^{d - 1}(V_{0, 1}, V_{\lambda, 1})_{\mathfrak{m}}^{\mathrm{ord}} = 0$ the Tor group vanishes and so
by the local criterion for flatness $ H_{?}^{d}(V_{1,\infty}, V_{\lambda, \mathcal{O}})_{\mathfrak{m}}^{\mathrm{ord}, \Delta^{p}}$ is a flat $\Lambda$-module. Since 
$$
H_{?}^{d}(V_{1,\infty}, V_{\lambda, 1})_{\mathfrak{m}}^{\mathrm{ord}, \Delta^{p}}/\mathfrak{m}_{1} \cong H_{?}^{d}(V_{0, 1}, V_{\lambda, 1})_{\mathfrak{m}}^{\mathrm{ord}}
$$ 
Nakayama gives us that $H_{?}^{d}(V_{1,\infty}, V_{\lambda, 1})_{\mathfrak{m}}^{\mathrm{ord}, \Delta^{p}}$ is finitely generated and thus it is free. We're left to showing 
$$
    H_{?}^{j}(V_{1,\infty}, V_{\lambda, 1})_{\mathfrak{m}}^{\mathrm{ord}, \Delta^{p}} = 0
$$
for $j < d$, but since $H_{?}^{d}(V_{1,\infty}, V_{\lambda, 1})_{\mathfrak{m}}^{\mathrm{ord}, \Delta^{p}}$ is free we have $(E^{i,d}_{2})^{\Delta^{p}} = 0$ for $i < 0$ so we can proceed much as we did for $j > d$.
\end{proof}

\begin{remark}
Examples of our assumption holding include the work of Mokrane-Tilouine \cite{mokrane}, where they prove that the assumptions hold for $G = \mathrm{GSp}_{2g}$ and suitably nice maximal ideals of $\mathbb{T}^{S}$. Such results will generally be proved at prime-to-$p$ level and can be deduced at level $V_{0,1}$ via Theorem \ref{thm:7}. 
\end{remark}

\begin{lemma} \label{lem:26}
Let $M$ be a $\Lambda_{0}$-module, then 
$$
    (M \otimes_{\Lambda} \Lambda/\mathfrak{m}_{1})^{\Delta^{p}} = M^{\Delta^{p}} \otimes_{\Lambda} \Lambda/\mathfrak{m}_{1}.
$$
\end{lemma}
\begin{proof}
We can decompose $M$ into a direct sum of $\psi$-eigenspaces $M^{(\psi)}$ where $\psi$ runs over characters $\psi: \Delta^{p} \to \mathcal{O}^{\times}$. Thus
$$
    M \otimes_{\Lambda} \Lambda/\mathfrak{m}_{1} = \oplus_{\psi}(M^{(\psi)} \otimes_{\Lambda} \Lambda/\mathfrak{m}_{1}).
$$
Since $\Delta^{p}$ acts on the tensor product through $M$ and preserves the eigenspaces by definition, we need only to show that $$
\psi(\Delta^{p}) \equiv 1 \ \mathrm{mod} \ p \implies \psi \equiv \mathbbm{1}
$$
but $\Delta^{p}$ is finite so $\psi$ takes values in $(\mathcal{O}/p)^{\times} \hookrightarrow \mathcal{O}^{\times}$, so the implication holds.
\end{proof}
\typeout{}

\providecommand{\bysame}{\leavevmode\hbox to3em{\hrulefill}\thinspace}
\providecommand{\MR}{\relax\ifhmode\unskip\space\fi MR }
\providecommand{\MRhref}[2]{%
  \href{http://www.ams.org/mathscinet-getitem?mr=#1}{#2}
}
\providecommand{\href}[2]{#2}

\bibliographystyle{amsalpha}
\end{document}